\documentclass{amsart}
\usepackage{amsthm}
\usepackage{amsmath,amssymb}
\usepackage{xcolor}
\usepackage{hyperref}
\newtheorem{theorem}{Theorem}

\newtheorem{prop}[theorem]{Proposition}
\newtheorem{lemma}[theorem]{Lemma}
\theoremstyle{remark}
\newtheorem{definition}{Definition}

\numberwithin{theorem}{section}
\newtheorem{rem}[theorem]{Remark}
\allowdisplaybreaks
\newcommand{\Om} {\Omega}

\newcommand{\be} {\begin{equation}}
\newcommand{\ee} {\end{equation}}
\newcommand{\bea} {\begin{eqnarray}}
\newcommand{\eea} {\end{eqnarray}}
\newcommand{\Bea} {\begin{eqnarray*}}
\newcommand{\Eea} {\end{eqnarray*}}

\def\R{{\mathbb R}}

\def\R{{\mathbb R}}

\numberwithin{equation}{section}
\begin{document}

\title[Kirchhoff equation with double criticality] { Kirchhoff type elliptic equations with double criticality in  Musielak-Sobolev spaces}
\author[Shilpa Gupta]
{Shilpa Gupta}
\address{Shilpa Gupta\hfill\break
Department of Mathematics\newline
 Birla Institute of Technology and Science Pilani \newline
  Pilani Campus, Vidya Vihar \newline
 Pilani, Jhunjhunu \newline
 Rajasthan, India - 333031}
 \email{p20180442@pilani.bits-pilani.ac.in; shilpagupta890@gmail.com }
 
 \author[Gaurav Dwivedi ]
{Gaurav Dwivedi}
\address{Gaurav Dwivedi \hfill\break
Department of Mathematics\newline
 Birla Institute of Technology and Science Pilani \newline
  Pilani Campus, Vidya Vihar \newline
 Pilani, Jhunjhunu \newline
 Rajasthan, India - 333031}
 \email{gaurav.dwivedi@pilani.bits-pilani.ac.in}

\begin{abstract}
This paper aims to establish the existence of  a weak solution for the non-local problem:
\begin{equation*}
\left\{\begin{array}{ll}
-a\left(\int_{\Omega}\mathcal{H}(x,|\nabla u|)dx \right) \Delta_{\mathcal{H}}u  &=f(x,u) \ \ \hbox{in} \ \ \Omega,  \ \ \ \\
\hspace{3.3cm} u  &= 0 \ \   \hbox{on}  \ \ \partial \Omega,
\end{array}\right.
\end{equation*}
where $\Omega\subseteq \R^{N},\, N\geq 2$ is a bounded and smooth domain containing two open and connected subsets $\Omega_p$ and  $\Omega_N$ such that $ \overline{\Omega}_{p}\cap\overline{\Omega}_{N}=\emptyset$ and $\Delta_{\mathcal{H}}u=\hbox{div}( h(x,|\nabla u|)\nabla u)$ is the $\mathcal{H}$-Laplace operator.  We assume that $\Delta_{\mathcal{H}}$ reduces to  $ \Delta_{p(x)}$ in $\Omega_{p}$ and to $ \Delta_{N}$ in $\Omega_{N},$  the  non-linear function $f:\Omega\times\R\rightarrow \R$  acts  as $|t|^{p^{\ast}(x)-2}t$ on $\Omega_{p}$ and as $e^{\alpha|t|^{N/(N-1)}}$ on $\Omega_{N}$ for sufficiently large $|t|$. To establish the existence results in a Musielak-Sobolev space, we use a variational technique based on  the mountain pass theorem.
\end{abstract}

\maketitle

\section{Introduction}
\setcounter{equation}{0}
This paper aims to establish the  existence of  a weak solution to the following non-local problem:
\begin{equation}\label{1.1}
\left\{\begin{array}{ll}
-a\left(\int_{\Omega}\mathcal{H}(x,|\nabla u|)dx \right) \Delta_{\mathcal{H}}u  &=f(x,u) \ \ \hbox{in} \ \ \Omega,  \ \ \ \\
\hspace{3.8cm} u  &= 0 \ \   \hbox{on}  \ \ \partial \Omega,
\end{array}\right.
\end{equation}
where $\Omega\subseteq \R^{N}$ is a  bounded and smooth domain,  $N\geq 2$, $\Delta_{\mathcal{H}}u=\hbox{div}( h(x,|\nabla u|)\nabla u),$  $\mathcal{H}(x,t)=\int_{0}^{|t|} h(x,s)s\ ds$ and $ h:\Omega\times [0,\infty)\rightarrow[0,\infty)$ is a generalized $N$-function. 
%The non-linear function $f:\Omega\times\R\rightarrow \R$ is continuous.
%For any  $h\in C(\overline{\Omega},(1,\infty))$, we denote $h^{-}=\min\limits_{x\in \Omega}h(x)$ and $h^{+}=\max\limits_{x\in \Omega}h(x).$ Throughout the paper, the functions $p,q,q_{1},p^{*}\in C(\overline{\Omega},(1,\infty))$.

Problem \eqref{1.1}  is known as  Kirchhoff type problem as it is related to the celebrated work of Kirchhoff \cite{kirchhoff1876vorlesungen}, where the author studied the equation:
\begin{equation}
\rho \frac{\partial^2 u}{\partial t^2}-\left(\frac{\rho_0}{h}+\frac{E}{2L}\int_\Om \left|\frac{\partial u}{\partial x}\right|^2dx\right) \frac{\partial^2 u}{\partial x^2}=0.
\end{equation}
When $\mathcal{H}(x,t)=t^p,$ \eqref{1.1} reduces to a Kirchhoff type problem for $p$-Laplace operator.  Lions \cite{lions1978} set-up an abstract framework for the study of such  problems and thereafter several authors obtained existence results for $p$-Kirchhoff type equations, see \cite{alves2021,alves2001,alves2005,chipot1997some,furtado2019,hamydy2011existence,naiman2017,sun2016,wang2018,radulescu1,yan2021} and references therein. If $\mathcal{H}(x,t)=t^{p(x)},$ \eqref{1.1} transforms into a Kirchhoff type problem with variable exponent and existence results are such problems are studied in the variable exponent Sobolev spaces. For some such results, one can refer to,  \cite{chung1, p_x,dai,dai1,hamdani,lee1} and references therein.  When $\mathcal{H}(x,t)$ is independent of $x,$ the existence results for problems of type \eqref{1.1} are discussed in Orlicz-Sobolev spaces, and we refer to the work of  Chaharlang and Razani \cite{razani}, and Chung \cite{Chung2} in this direction. In the case, when $\mathcal{H}(x,t)$ depends on both $x$ and $t,$ the existence of a solution for the problems of the type \eqref{1.1} is studied in Musielak-Sobolev spaces (for the definitions and properties of variable exponent Sobolev spaces and Musielak-Sobolev spaces, see Section \ref{sec2}). The study of Musielak spaces started in the mid-1970s with the work of Musielak \cite{musielak} and Hudzik \cite{hudzik,hudzik1},  where the authors provide the general framework for Musielak spaces in terms of modular function.  Many authors \cite{subsup,gossez,hachimi,shi} used such spaces to prove the existence of a solution for problems of the type \eqref{1.1} without the Kirchhoff term $a.$ Shi and Wu \cite{shi} studied the existence result for Kirchhoff type problems in Musielak-Sobolev spaces. Chlebicka \cite{pocket} provides an extensive survey of elliptic partial differential equations in Musielak spaces.  Recently, Alves et al. \cite{double} developed the concept of double criticality and studied the quasilinear problem in Musielak-Sobolev spaces, and our existence results are motivated by their work.

Next, we state our hypotheses. Throughout this article, for any  $r\in C(\overline{\Omega},(1,\infty))$, we denote $r^{-}=\min\limits_{x\in \Omega}r(x)$ and $r^{+}=\max\limits_{x\in \Omega}r(x).$ Further, the functions $p,q,p^{*},q_{1}\in C(\overline{\Omega},(1,\infty))$.
We consider the following assumptions on the functions $\mathcal{H}$ and $ h:$
 \begin{itemize}
 \item[$ (\mathcal{H}_{1}) $]   $ h(x,\cdot)\in C^{1}$ in $(0,\infty)$,  $\forall x\in\Omega$.
 \item[$ (\mathcal{H}_{2}) $]  $ h(x,t),\, \partial_{t}( h(x,t)t)>0,$ $\forall x\in\Omega$ and $t>0$.
 \item[$ (\mathcal{H}_{3}) $]  $p^{-}\leq\frac{ h(x,|t|)|t|^{2}}{\mathcal{H}(x,|t|)}\leq q^{+}$ for $x\in \Omega$ and  $t\neq 0$  for some $1<p^{-}\leq p(x)<N<q(x)\leq q^{+}<(p^{*})^{-}$.
\item[$ (\mathcal{H}_{4}) $]  $\inf\limits_{x\in\Omega}\mathcal{H}(x,1)=b_{1}$ for some $b_{1}>0$.
\item[$ (\mathcal{H}_{5}) $]  For each $t_{0}\neq 0$, there exists $d_{0}>0$ such that $\frac{\mathcal{H}(x,t)}{t}\geq d_{0}$ and  $\frac{\widetilde{\mathcal{H}}(x,t)}{t}\geq d_{0}$ for $t\geq t_{0}$ and $x\in\Omega$,  where $\widetilde{\mathcal{H}}(x,t)=\int_{0}^{|t|}\widetilde{ h}(x,s)s\ ds$,  $\widetilde{ h}$ is the complimentary function of $ h$ which is defined as $\widetilde{ h}(x,t)=\sup\{s: h(x,s)s\leq t\} \ \forall(x,t)\in\overline{\Omega}\times [0,\infty).$
\end{itemize}
Let $S\subset\Omega$ and $\delta>0$. The $\delta$ neighborhood of $S$ is denoted by $S_{\delta}$ and defined as 
$$S_{\delta}=\{x\in\Omega:\text{dist}(x,S)<\delta \}.$$

Assume that, we have three smooth domains $ \Omega_{p}, \ \Omega_{N}$ and $\Omega_{q}$ with non-empty interiors  such that
$\Omega=\Omega_{p}\cup\Omega_{N}\cup\Omega_{q}$ and $(\overline{\Omega}_{p})_{\delta}\cap(\overline{\Omega}_{N})_{\delta}=\emptyset$.

 Next, we define continuous functions $ \psi_{p},\ \psi_{N},\ \psi_{q}:\overline{\Omega}\rightarrow[0,1]$ such that
 $$ \psi_{p}(x)=1 \ \forall x\in \overline{\Omega_{p}}, \ \   \psi_{p}(x)=0  \ \forall x\in (\overline{\Omega_{p}})^{c}_{\delta},$$
  $$ \psi_{N}(x)=1 \ \forall x\in \overline{\Omega_{N}}, \ \  \psi_{N}(x)=0  \ \forall x\in (\overline{\Omega_{N}})^{c}_{\delta},$$
  $$ \psi_{q}(x)=1 \ \forall x\in \overline{\Omega_{q}},  \ \   \psi_{q}(x)=0  \ \forall x\in (\overline{\Omega_{q}})^{c}_{\delta}.$$
  
We consider that the non-linear function $f:\Omega\times\R\rightarrow \R$ is continuous and of the following type:
\begin{equation}\tag{$f_{1}$}
f(x,t)=\lambda \psi_{N}(x)|t|^{\beta-2}te^{\alpha|t|^{\frac{N}{N-1}}}+\widetilde{ \psi}_{q}(x)\varphi(x,t)+ \psi_{p}(x)|t|^{p^{\ast}-2}t \ \ \forall (x,t)\in \Omega\times \R,
\end{equation}
where $(p^{*})^{+}\geq p^{*}(x)\geq(p^{*})^{-}>q^{+}\geq q(x)\geq q^{-}>N>p^{+}\geq p(x)\geq p^{-}>N/2$, $\beta>q^{-}$, $\lambda>0$ and $\alpha>0$. Moreover, $\widetilde{ \psi}_{q}:\overline{\Omega}\rightarrow [ 0,1],$
 $\varphi:\overline{\Omega}\times\R\rightarrow\R$ are continuous functions such that
$$\widetilde{ \psi}_{q}(x)=1, \ \forall x\in \Omega_{q} \  \text{and} \ \widetilde{ \psi}_{q}(x)=0,  \ \forall x\in (\overline{\Omega}_{q})^{c}_{\delta/2},$$
and
$\varphi(x,t)=o(|t|^{q_{1}(x)-1}) $ as $t\rightarrow 0$ uniformly on $(\overline{\Omega}_{q})_{\delta/2}$ for some $q_{1}^{+}\geq q_{1}(x)\geq q_{1}^{-}>q^{-},$ and there exists $\chi>q^{-}$ such that 
$$0<\chi \Phi(x,t)\leq \varphi(x,t)t, \ \forall x\in (\overline{\Omega}_{q})_{\delta/2},$$ where $\Phi(x,t)=\int_{0}^{|t|}\varphi(x,s)ds.$\\

Along with the above notations, $h$  also satisfies the following conditions for each $t>0$:
 \begin{itemize}
 \item[$ (\mathcal{H}_{6}) $]   $ h(x,t)\geq t^{N-2} \ \ \forall x\in\Omega_{N}$ and $C_{1} t^{N-2}\geq  h(x,t) \ \forall x\in\Omega_{N}\backslash (\overline{\Omega}_{q})_{\delta} $ for some $C_{1}>0$.
 \item[$ (\mathcal{H}_{7}) $] There exist a  continuous function $ \eta_{1}:\overline{\Omega}\rightarrow \R$ such that $ h(x,t)\geq  \eta_{1}(x)  t^{q(x)-2}$ $\forall x\in(\Omega_{q})_{\delta}$ and  $ \eta_{1}(x)>0, \ \forall x\in(\Omega_{q})_{\delta}$ and $ \eta_{1}(x)=0, \ \forall x\in((\Omega_{q})_{\delta})^{c}.$
 \item[$ (\mathcal{H}_{8}) $]  There exist a non-negative continuous function $ \eta_{2}:\overline{\Omega_{p}}\rightarrow \R$ such that $ \eta_{2}(x)  t^{q(x)-2}+C_{2}t^{p(x)-2}\geq h(x,t)\geq t^{p(x)-2} \  \ \forall x\in\Omega_{p}$ and   $ \eta_{2}(x)>0, \ \forall x\in(\Omega_{q})_{\delta}$ and $ \eta_{2}(x)=0, \ \forall x\in\overline{\Omega_{p}}\backslash(\overline{\Omega_{q})_{\delta}},$  for some $C_{2}>0$.
\end{itemize}

Next, we state our hypotheses on the nonlocal term $a.$ The continuous function $a: \R^{+} \rightarrow\R^{+}$  satisfies following conditions: 
\begin{itemize}
\item[$(a_{1})$]  There exist positive real number $a_{0}$ such that 
$ a(s) \geq a_{0}$  and $a$ is non-decreasing $\forall \ s>0$.
\item[$(a_{2})$]  There exist  $\theta>1$ such that $\beta >N\theta$ and $a(s)/s^{\theta-1}$ is non-increasing for $s>0.$
\end{itemize}
\begin{rem}
By $(a_{2})$, we have
\begin{itemize}
\item[$(a'_{2})$]  
$\theta A(s)-a(s)s$ is non-decreasing, $\forall \ s>0$, where $A(s)=\int_{0}^{s}a(t)dt$.
\end{itemize}
In particular, 
\begin{equation}\label{a1}
\theta A(s)-a(s)s\geq 0 \ \  \forall \ s>0.
\end{equation}
Again by $(a_{2})$ and \eqref{a1}, one gets
\begin{itemize}
\item[$(a''_{2})$]  
$ A(s)\leq s^{\theta}A(1)$  $\forall \ s\geq 1$.
\end{itemize}
\end{rem}

Now, we state main result of this article:
  \begin{theorem}\label{t1}
  Suppose  that the conditions $(f_{1}),$ $(\mathcal{H}_{1})-(\mathcal{H}_{8})$ and $(a_{1})-(a_{2})$ are satisfied. Then there exists $\lambda_{1}>0$ such that  for any $ \lambda\geq\lambda_{1}$,    Problem (\ref{1.1}) has non trivial weak solution  via mountain pass theorem.
  \end{theorem}

 This article is organized as follows:  We discuss the definition and properties of Musielak-Sobolev spaces and the functional setup needed to prove our result in Section \ref{sec2}. Section 3 deals with the proof of Theorem \ref{t1}.
 
\section{Preliminaries}\label{sec2}

In this section, we discuss Musielak spaces and their properties. We also provide the functional setup needed to prove our main result and discuss  some helping results. Define, 
\[\mathcal{H}(x,t)=\int_{0}^{|t|} h(x,s)s\ ds,\] where $ h:\Omega\times [0,\infty)\rightarrow[0,\infty).$ Then $\mathcal{H}(x,t)$ is a generalized $N$-function. Recall that, $\mathcal{H}(x,t):\Omega\times [0,\infty)\rightarrow[0,\infty)$ is said to be a generalized $N$-function if it is  continuous, even, convex function of $t$, $\mathcal{H}(x,t)=0$ iff $t=0$,  $\lim\limits_{t\rightarrow 0}\frac{\mathcal{H}(x,t)}{t}=0$ and $\lim\limits_{t\rightarrow \infty}\frac{\mathcal{H}(x,t)}{t}=\infty$.

The Musielak space $L^{\mathcal{H}}(\Omega)$ is defined as:
$$L^{\mathcal{H}}(\Omega)=\left\lbrace u:\Omega\rightarrow\R \ \text{is measurable function}\left| \  \int_{\Omega}\mathcal{H}\left( x,\tau|u|\right) dx<\infty, \  \text{for some} \  \tau>0\right\rbrace \right..$$  $L^{\mathcal{H}}(\Omega)$ is a reflexive Banach space \cite{musielak}  with the Luxemburg norm
$$\| u\|_{L^{\mathcal{H}}(\Omega)}=\inf\left\lbrace \tau>0\left| \ \int_{\Omega}\mathcal{H}\left( x,\frac{|u|}{\tau}\right) dx\leq 1\right\rbrace \right. \cdot$$

We say that a generalised $N$-function satisfies the weak $\Delta_{2}$-condition if there exist $C>0$ and a non-negative function $k\in L^{1}(\Omega)$ such that 
$$\mathcal{H}(x,2t)\leq C\mathcal{H}(x,t)+k(x) \ \ \forall (x\times t)\in\Omega\times\R.$$ If $k=0,$ then $\mathcal{H}$ is said to satisfy  $\Delta_{2}$-condition. Also, the function $\mathcal{H}$ and its complementary function $\widetilde{\mathcal{H}}$( defined in $ (\mathcal{H}_{5}) $)   satisfy the following  Young's inequality \cite[Proposition 2.1]{liu2015}: 
$$s_{1}s_{2}\leq\mathcal{H}(x,s_{1})+\widetilde{\mathcal{H}}(x,s_{2}) \ \forall x\in \Omega, s_{1},s_{2}>0.$$
Further, proceeding as \cite[Lemma A2]{fukagai2006}, we have 
\begin{equation}\label{comp}
\widetilde{\mathcal{H}}(x, h(x,s)s)\leq\mathcal{H}(x,2s), \ \ \forall\,(x,s)\in\overline{\Omega}\times [0,\infty).
\end{equation}

 The Musielak-Sobolev space $W^{1,\mathcal{H}}(\Omega)$ is defined as
 $$W^{1,\mathcal{H}}(\Omega)=\{u\in L^{\mathcal{H}}(\Omega)\left| \  |\nabla u|\in L^{\mathcal{H}}(\Omega)\right.\}.$$
  $W^{1,\mathcal{H}}(\Omega)$  is  a  Banach space with the  norm \cite[Theorem 10.2]{musielak} 
 $$\| u\|_{1,\mathcal{H}}=\| u\|_{L^{\mathcal{H}}(\Omega)}+\| \nabla u\|_{L^{\mathcal{H}}(\Omega)}.$$ 
The space $W^{1,\mathcal{H}}_{0}(\Omega)$ is defined as the closure of $C_{c}^{\infty}(\Omega)$ in $W^{1,\mathcal{H}}(\Omega).$  Also, the space $W^{1,\mathcal{H}}_{0}(\Omega)$ is  equipped with the norm  $\| u\|=\| \nabla u\|_{L^{\mathcal{H}}(\Omega)},$ which is equivalent to the norm $\| \cdot\|_{1,\mathcal{H}}$ \cite[Lemma 5.7]{gossez}.  
\begin{theorem}\cite{musielak}
The spaces $L^{\mathcal{H}}(\Omega)$ and $W^{1,\mathcal{H}}_{0}(\Omega)$ are  reflexive and separable Banach spaces.
\end{theorem}
In particular, if we take $\mathcal{H}(x,t)=t^{p(x)}$ then we denote  $L^{\mathcal{H}}(\Omega)$ as $L^{p(x)}(\Omega)$ and $W^{1,\mathcal{H}}(\Omega)$ as $W^{1,p(x)}(\Omega)$ . Such spaces are called variable exponent Lebesgue and variable exponent Sobolev spaces, respectively. To know more about these spaces, one can check \cite{cruz2013variable,fan, radulescu}. 

Further, we have the following embedding result:
\begin{prop}\label{prop3}\cite{fan}
Let $\Omega$ be a bounded smooth domain. Then the following embeddings are continuous:
\begin{enumerate}
\item[$(a)$]   $W^{1,\mathcal{H}}_{0}(\Omega)\hookrightarrow L^{\gamma}(\Omega_{N}), \ 1\leq\gamma <\infty,$

\item[$(b)$]   $W^{1,\mathcal{H}}_{0}(\Omega)\hookrightarrow L^{s(x)}((\Omega_{p})_{\delta}),$ where $s(x)\leq\frac{Np(x)}{N-p(x)}$.
\item[$(c)$]    $W^{1,\mathcal{H}}_{0}(\Omega)\hookrightarrow W^{1,q^{-}}_{0}((\Omega_{q})_{\delta}),$
\end{enumerate}
Moreover,  the embedding
\begin{equation}\label{embed}
W^{1,\mathcal{H}}_{0}(\Omega)\hookrightarrow C(\overline{(\Omega_{q})_{\delta}}) \hbox{ is compact }. 
\end{equation}
\end{prop}
\begin{proof}
By using the conditions $(\mathcal{H}_{6}),(\mathcal{H}_{7}),(\mathcal{H}_{8})$ and the definition of $W^{1,\mathcal{H}}_{0}(\Omega)$,  we have continuous embeddings 
 $$W^{1,\mathcal{H}}_{0}(\Omega)\hookrightarrow W^{1,N}_{0}(\Omega_{N})$$
$$W^{1,\mathcal{H}}_{0}(\Omega)\hookrightarrow W^{1,q(x)}_{0}((\Omega_{q})_{\delta}),$$
$$W^{1,\mathcal{H}}_{0}(\Omega)\hookrightarrow W^{1,p(x)}_{0}((\Omega_{p})_{\delta}).$$
  Further,  $W^{1,N}_{0}(\Omega_{N})\hookrightarrow L^{\gamma}(\Omega_{N})$ is continuous for any $1\leq\gamma <\infty$ \cite[Theorem 2.4.4]{kesawan}, which proves $(a)$. 

We know that $W^{1,p(x)}_{0}((\Omega_{p})_{\delta})\hookrightarrow L^{s(x)}((\Omega_{p})_{\delta})$ is continuous for $s(x)\leq\frac{Np(x)}{N-p(x)}$ \cite[Theorem 2.3]{fan}. This proves $(b).$

 For  $(c)$, as $q^{-}\leq q(x)$,  $W^{1,q(x)}_{0}((\Omega_{q})_{\delta})\hookrightarrow W^{1,q^{-}}_{0}((\Omega_{q})_{\delta})$ is continuous. 
 Moreover, since  $q^{-}>N,$  $W^{1,q^{-}}_{0}((\Omega_{q})_{\delta})\hookrightarrow C(\overline{(\Omega_{q})_{\delta}})$is compact\cite[Theorem 2.5.3]{kesawan}  and  this implies that $W^{1,\mathcal{H}}_{0}(\Omega)\hookrightarrow C(\overline{(\Omega_{q})_{\delta}})$ is compact. 
  \end{proof}

Next, we will state some results which are used to prove our main result.
\begin{prop}\label{prop4}\cite[Proposition 1.5]{subsup}
Let $\mathcal{H}$ be a generalized $N$-function. If $ (\mathcal{H}_{4}) $ holds then  $L^{\mathcal{H}}(\Omega)\hookrightarrow L^{1}(\Omega)$ and $W^{1,\mathcal{H}}(\Omega)\hookrightarrow W^{1,1}(\Omega)$.
\end{prop}

\begin{prop}\label{prop5}\cite[Theorem 2.1]{kovacik}
Let $r\in C(\overline{\Omega},(1,\infty))$ and $s\in C(\overline{\Omega},(1,\infty))$ be the conjugate exponent of $r$. Then, for any $u\in L^{r(x)}(\Omega)$ and $v\in L^{s(x)}(\Omega)$, we have
$$\left| \int_{\Omega}uv\ dx\right|\leq\left(\frac{1}{r^{-}}+\frac{1}{s^{-}}\right) \|u\|_{L^{r(x)}(\Omega)}\|v\|_{L^{s(x)}(\Omega)}.$$
\end{prop}

\begin{prop}\label{prop2}\cite{fan}
For any  $u\in L^{p(x)}(\Omega),$ the followings are true:
\begin{enumerate}
\item   $\|u\|_{L^{p(x)}(\Omega)}^{p^{-}}\leq \rho(u)\leq\|u\|_{L^{p(x)}(\Omega)}^{p^{+}}$ whenever ${\|u\|_{L^{p(x)}(\Omega)}}> 1,$
\item   $\|u\|_{L^{p(x)}(\Omega)}^{p^{+}}\leq \rho(u)\leq\|u\|_{L^{p(x)}(\Omega)}^{p^{-}}$ whenever ${\|u\|_{L^{p(x)}(\Omega)}}<1,$
\item  ${\|u\|_{L^{p(x)}(\Omega)}}<1(=1;>1)$ iff  $\rho(u)<1(=1;>1)$,
\end{enumerate}
where $\rho(u)=\int_\Omega |u|^{p(x)}dx.$
\end{prop}

\begin{prop}\label{prop6}\cite{kovacik}
Let $r,s\in C(\overline{\Omega},(1,\infty))$ such that $1<r(x)s(x)<\infty$. Then, for any  $u\in L^{s(x)}(\Omega),$ the followings are true:
\begin{enumerate}
\item   $\|u\|_{L^{r(x)s(x)}(\Omega)}^{r^{-}}\leq\||u|^{r(x)}\|_{L^{s(x)}(\Omega)}\leq\|u\|_{L^{r(x)s(x)}(\Omega)}^{r^{+}}$ whenever ${\|u\|_{L^{r(x)s(x)}(\Omega)}}\geq 1,$
\item   $\|u\|_{L^{r(x)s(x)}(\Omega)}^{r^{+}}\leq\||u|^{r(x)}\|_{L^{s(x)}(\Omega)}\leq\|u\|_{L^{r(x)s(x)}(\Omega)}^{r^{-}}$ whenever ${\|u\|_{L^{r(x)s(x)}(\Omega)}}\leq 1.$
\end{enumerate}
\end{prop}
 Define the function $m:W^{1,\mathcal{H}}_{0}(\Omega)\rightarrow\R$ as
$$m(u)=\int_\Omega\mathcal{H}(x,|\nabla u|)dx.$$
\begin{prop}\label{prop1}\cite{double}
For any $u\in W^{1,\mathcal{H}}_{0}(\Omega),$  the followings are true:
\begin{enumerate}
\item   $\|u\|^{p^{-}}\leq m(u)\leq \|u\|^{q^{+}}$ whenever  $\|u\|\geq 1$.
\item   $\|u\|^{q^{+}}\leq m(u)\leq \|u\|^{p^{-}}$ whenever $\|u\|\leq 1$.
\end{enumerate}
In particular, $m(u)=1$ iff $\|u\|=1$. Moreover, if $\{u_{n}\}\subset W^{1,\mathcal{H}}_{0}(\Omega),$ then $\|u_{n}\|\rightarrow 0$ iff $m(u_{n})\rightarrow 0.$
\end{prop}

\begin{lemma}\label{l8}\cite[Theorem 2.2]{subsup}
Suppose that $(\mathcal{H}_{1}) - (\mathcal{H}_{8})$ hold.  If $u_{n}\rightharpoonup u$ in $W^{1,\mathcal{H}}_{0}(\Omega)$ and  $$\lim_{n\rightarrow\infty} \int_{\Omega} h(x,|\nabla u_{n}|) \ \nabla u_{n} \ \nabla (u_{n}-u)\leq 0,$$ then $u_{n}\rightarrow  u$ in $W^{1,\mathcal{H}}_{0}(\Omega)$.
\end{lemma} 

Next, we discuss some properties of the nonlinear function $f.$
We assume that the nonlinear function $f$ has exponential type growth on $\Omega_{N},$ which is  motivated by the celebrated result  of  N. Trudinger \cite{trudinger}. N. Trudinger \cite{trudinger} proved that $W^{1,N}_{0}(\Omega)$ is continuously embedded in the Orlicz space $L_\mathcal{H}(\Omega),$ where $\mathcal{H}=\exp(t^{\frac{n}{n-1}})-1.$ The inequality of Trudinger was later sharpened by J. Moser \cite{moser} and known as Moser-Trudinger inequality. In the subsequent years, many authors  improved and used the Moser-Trudinger inequality to study the  problems involving exponential type non-linearities. Interested readers can refer to \cite{adi,Adimurthi,do2,do1,figu2019,Lam,lam2014} and references cited therein.
We will use the following version of Moser-Trudinger inequality:
\begin{lemma}\label{l7}\cite[Lemma 3.4]{double}
Let   $\alpha>0$ and $s>1$ then $\exists$ $0< r< 1$ and $C>0$ such that 
  $$\sup\int_{\Omega}e^{s\alpha |u|^{\frac{N}{N-1}}}dx\leq C,$$ 
  for any  $u\in W^{1,\mathcal{H}}_{0}(\Omega)$ such that  $\|u\|\leq r$.
\end{lemma}

%We use the mountain pass theorem  to prove our existence result. For convenience of the reader, we recall the following definition:
%\begin{definition}(Palais-Smale condition) \normalfont  Let $V$ be a Banach space and  $I:V\rightarrow \R $ be a $C^{1}$ functional.  We say that $I$ holds the Palais-Smale condition $(PS)_{c}$ (where $c\in \R$ is a constant) if every sequence $\{u_{n}\}$ in $V$ such that $I(u_{n})\rightarrow c$ and $I'(u_{n})\rightarrow 0$ in the dual space $V^{*},$ has a convergent subsequence.
%\end{definition}

We assume that the nonlinear function $f$ has critical growth on $\Omega_{p}$, which causes a lack of compactness and hence, one can not prove the Palais-Smale condition directly. Lions established  concentration compactness principle \cite[Lemma 1.1]{plion} to address such issues. We use the following variable exponent version of the concentration compactness principle that was obtained by Bonder and Silva \cite{silva}.

\begin{lemma}\label{cc}
Let $\{u_{n}\}$   in $W^{1,p(x)}_{0}(\Omega)$ which converges weakly to  limit $u$  such that
\begin{itemize}
\item $|\nabla u_{n}|^{p(x)}$ converges weakly to a measure $\mu$,
\item   $|u_{n}|^{p^{*}(x)}$ converges weakly to a measure $\nu$, where $\mu$ and $\nu$ are bounded non-negative measures on $\Omega$. 
\end{itemize}
Then there exist atmost countable index set $I$ and  $(x_{i})_{i\in I}\in \Omega$  such that 
 \begin{enumerate}
 \item [(1)] $\nu=|u|^{p^{*}(x)}+\sum\limits_{i\in I}\nu_{i}\delta_{x_{i}}, \ \nu_{i}>0$
 \item [(2)] $\mu\geq|\nabla u|^{p(x)}+\sum\limits_{i\in I}\mu_{i}\delta_{x_{i}}, \ \mu_{i}>0$ \\
 with $S\nu_{i}^{1/p^{*}(x_{i})}\leq\mu_{i}^{1/p(x_{i})}, \forall i\in I$,
 \end{enumerate}
 where $$S=\inf_{u\in C_{c}^{\infty}(\Omega)}\left\lbrace\dfrac{\||\nabla u|\|_{L^{p(x)}(\Omega)}}{\| u\|_{L^{p^{*}(x)}(\Omega)}} \right\rbrace>0.$$
\end{lemma}

Next, we define a weak solution to \eqref{1.1} and the corresponding energy functional. 
\begin{definition}\normalfont
We say that  $u\in W^{1,\mathcal{H}}_{0}(\Omega)$ is a weak solution of  \eqref{1.1} if the following holds:
\begin{equation}\label{wf2}
a\left( m(u)\right) \int_{\Omega} h(x,|\nabla u|) \ \nabla u \ \nabla v= \int_{\Omega} f(x,u) v
\end{equation}
for all $v\in W^{1,\mathcal{H}}_{0}(\Omega).$
\end{definition}
Thus, the energy functional $J:W^{1,\mathcal{H}}_{0}(\Omega)\rightarrow \R$ corresponding to \eqref{wf2} is given by 
$$J(u)=A\left( m(u)\right)  -\int_{\Omega} F(x,u) \ dx ,$$
where $F(x,t)=\int_{0}^{t}f(x,s) ds$ and $A(t)=\int_{0}^{t}a(s)ds$. It can be seen that $J$ is $C^{1}$ \cite[Lemma 3.8]{double} and the derivative of $J$ at any point $u\in W^{1,\mathcal{H}}_{0}(\Omega)$ is given by
$$J'(u)(v)= a\left( m(u)\right) \int_{\Omega} h(x,|\nabla u|) \ \nabla u \ \nabla v-\int_{\Omega} f(x,u) v$$
for all $v\in W^{1,\mathcal{H}}_{0}(\Omega).$
Moreover, the critical points of $J$ are the weak solutions to  \eqref{1.1}.

\section{Proof of the Theorem \ref{t1}}

To prove our main result, we first establish a series of lemmas.
\begin{lemma}\label{l1}
There exist positive real numbers $\alpha$ and $\rho$ such that for each $\lambda\geq 1$ we have
$$J(u)\geq \alpha>0, \ \ \forall u \in  W^{1,\mathcal{H}}_{0}(\Omega):\|u\|=\rho.$$
\end{lemma}
\begin{proof}
It follows, from  the definition of $f$ that
\begin{equation}
\int_{\Omega}F(x,t)dx=\int_{(\Omega_{q})_{\delta/2}}F(x,u)dx+\lambda\int_{\Omega_{N}\backslash (\Omega_{q})_{\delta/2}}F_{1}(x,u)dx+\int_{\Omega_{p}\backslash (\Omega_{q})_{\delta/2}}\frac{|u|^{p^{*}(x)}}{p^{*}(x)}dx
\end{equation}
where, $F_{1}(x,t)=\int_{0}^{t}|s|^{\beta-2}se^{\alpha|s|^{N/(N-1)}}ds.$
Again, from the definition of $f$, we get
\begin{equation*}
\int_{(\Omega_{q})_{\delta/2}}F(x,u)dx\leq c_{1}\int_{(\Omega_{q})_{\delta/2}}(|u|^{q_{1}(x)}+|u|^{\beta}+|u|^{p^{*}(x)}),
\end{equation*}
for $\|u\|=r$, where $r<1$ is small enough and for some $c_{1}>0$. Using  \eqref{embed} and the fact that $\|u\|=r$, where $r<1$ is small enough, one gets
\begin{equation}\label{e3}
\begin{split}
\int_{(\Omega_{q})_{\delta/2}}F(x,u)dx&\leq c_{2}(\|u\|_{L^{q_{1}^{-}}(\Omega_{q})_{\delta/2}}^{q_{1}^{-}}+\|u\|^{\beta}+\|u\|_{L^{(p^{*})^{-}}(\Omega_{q})_{\delta/2}}^{(p^{*})^{-}})\\
&\leq c_{3}(\|u\|^{q_{1}^{-}}+\|u\|^{\beta}+\|u\|^{(p^{*})^{-}})
\end{split}
\end{equation}
for some $c_{2},c_{3}>0$.

Next, by using the H$\ddot{\text{o}}$lder's inequality, one gets
\begin{equation*}
\lambda\int_{\Omega_{N}\backslash (\Omega_{q})_{\delta/2}}F_{1}(x,u)dx\leq \lambda\left(  \int_{\Omega_{N}}|u|^{2\beta}\right)^{\frac{1}{2}}\left(  \int_{\Omega_{N}}e^{2\alpha|u|^{\frac{N}{N-1}}}\right)^{\frac{1}{2}}.
\end{equation*}
Letting, $\|u\|=r<1$, by  Proposition \ref{prop3} $(a)$ and Lemma \ref{l7}, we obtain
\begin{equation}\label{e4}
\lambda\int_{\Omega_{N}\backslash (\Omega_{q})_{\delta/2}}F_{1}(x,u)dx\leq c_{4}\|u\|^{\beta},
\end{equation}
for some $c_{4}>0$.

Again, using Proposition \ref{prop3} $(b)$ and Proposition \ref{prop2}, we get 
\begin{equation}\label{e5}
\int_{\Omega_{p}\backslash (\Omega_{q})_{\delta/2}}\frac{|u|^{p^{*}(x)}}{p^{*}(x)}dx\leq c_{5}\|u\|^{(p^{*})^{-}},
\end{equation}
for $\|u\|=r$, where $r<1$ and  $c_{5}>0$. 
By the help of \eqref{e3}, \eqref{e4}, \eqref{e5}, $(a_{1})$ and the Proposition \ref{prop1}, we have
$$J(u)\geq a_{0}\|u\|^{q^{+}}-c_{6}\|u\|^{q_{1}^{-}}-c_{7}\|u\|^{\beta}-c_{8}\|u\|^{(p^{*})^{-}},$$ for some $c_{6},c_{7},c_{8}>0$. 
We can conclude the result by the fact that $q_{1}^{-}, (p^{*})^{-},\beta>q^{+}.$
\end{proof}

\begin{lemma}\label{l2}
 There exist $\nu_{0}\in  W^{1,\mathcal{H}}_{0}(\Omega)$ and   $\beta>0$  such that for each $\lambda\geq 1,$ we have
$$J(\nu_{0})< 0 \ \ \hbox{and} \ \|\nu_{0}\|>\beta.$$ 
\end{lemma}

\begin{proof}
By the definition of $f$ and as $\lambda\geq 1$, we get
\begin{equation}\label{theta}
f(x,s)\geq |s|^{\beta-2}s, \ \ \forall \ (x,s)\in (\Omega_{N}\backslash \overline{(\Omega_{q})_{\delta}})\times [0,\infty).
\end{equation}
Let $u\in C_{c}^{\infty}(\Omega_{N}\backslash \overline{(\Omega_{q})_{\delta})}\backslash\{0\}$ with $\|u\|=1$, using $(\mathcal{H}_{6})$, $(a''_{2})$ and  \eqref{theta}, we have
\begin{equation*} 
\begin{split}
J(tu)& =A\left( m(tu)\right)  -\int_{\Omega} F(x,tu)dx \\
 & \leq {A(1)}\left( \int_{\Omega_{N}}\mathcal{H}(x,|\nabla tu|)dx\right)^{\theta}-\frac{ |t|^{\beta}}{\beta}\int_{\Omega_{N}}|u|^{\beta}dx, \  \forall \  t>1,\\
 & \leq {A(1)}c_{1}t^{N\theta}\left( \int_{\Omega_{N}}|\nabla u|^{N}\right)^{\theta}-\frac{ |t|^{\beta}}{\beta}\int_{\Omega_{N}}|u|^{\beta}dx, \  \forall \  t>1,
\end{split}
\end{equation*}
this implies that $J(tu)\rightarrow -\infty$ as $n\rightarrow\infty$, since $\beta>N\theta$. Now, by setting $\nu_{0}=t_{0}u$ for sufficiently large $t_{0}>1$, we get the desired result. 
\end{proof}

By  Lemmas \ref{l1} and \ref{l2}, the geometric conditions of the mountain pass theorem are satisfied for the functional $J$.  Hence, by the version of the mountain pass theorem without (PS) condition, $\exists$ a sequence  $\{u_{n}\}\subseteq W^{1,\mathcal{H}}_{0}(\Omega)$ such that  $J(u_{n})\rightarrow c_{M}$ and $J'(u_{n})\rightarrow 0$ as  $n\rightarrow \infty$, where
$$c_{M}=\inf_{\gamma\in \varGamma}\max_{t\in [0,1]}J(\gamma(t))>0,$$
and
$$\varGamma=\{\gamma\in C([0,1],W^{1,\mathcal{H}}_{0}(\Omega)):\gamma(0)=0, \ \gamma(1)<0\}.$$
Due to the lack of compactness, we are not able to prove that (PS) condition holds for $J$ and we need some additional information about the mountain pass level $c_{M}$.

 \begin{lemma}\label{l3}
  The $(PS)_{c_{M}}$ sequence is bounded in $W^{1,\mathcal{H}}_{0}(\Omega)$. Moreover, there exist $u_{0}\in W^{1,\mathcal{H}}_{0}(\Omega)$ such that, up to a subsequence, we have
  $u_{n}\rightharpoonup u_{0}$ weakly in $W^{1,\mathcal{H}}_{0}(\Omega)$ and $u_{n}(x)\rightarrow u_{0}(x)$ a.e. $x\in\Omega$.

 \end{lemma}
 \begin{proof}
 If $ \psi=\min\{\chi,\beta,(p^{*})^{-}\}$, we have
\begin{equation}\label{p1}
0< \psi F(x,t)\leq f(x,t)t, \ \forall(x,t)\in \Omega\times(\R\backslash\{0\}).
\end{equation}
Since  $\{u_{n}\}$ is a $(PS)_{c_{M}}$ sequence for $J$, we have
$J(u_{n})\rightarrow c_{M}$ and $J'(u_{n})\rightarrow 0$ as  $n\rightarrow \infty$, i.e.,
\begin{equation}\label{p11}
A\left( m(u_{n})\right)  -\int_{\Omega} F(x,u_{n}) \ dx=c_{M}+\delta_{n},
\end{equation}
where $\delta_{n}\rightarrow 0$ as $n\rightarrow\infty$
and 
\begin{equation}\label{p22}
\left| a\left( m(u_{n})\right) \int_{\Omega} h(x,|\nabla u_{n}|) \ \nabla u_{n} \ \nabla v dx -\int_{\Omega} f(x,u_{n}) v\right|  \leq \varepsilon_{n}\|v\|,
\end{equation}
$\forall\, v\in W^{1,\mathcal{H}}_{0}(\Omega),$ where $\varepsilon_{n}\rightarrow 0$ as $n\rightarrow\infty$.
On taking $v=u_n,$ by using \eqref{p1}, \eqref{p11} and \eqref{p22}, we obtain
\begin{equation*}
\begin{split}
A\left( m(u_{n})\right)-\frac{1}{ \psi}a\left( m(u_{n})\right) &\int_{\Omega} h(x,|\nabla u_{n}|) \ |\nabla u_{n}|^{2} \\
&\leq c_{9}(1+\|u_{n}\|),
\end{split}
\end{equation*}
for some $c_{9}>0$. It follows from $(a_{1})$  that
\begin{equation*}
a_{0}\left(1-\frac{q^{+}}{ \psi}\right) m(u_{n})
\leq c_{9}(1+\|u_{n}\|).
\end{equation*}
If $\|u\|\geq 1$, by Proposition \ref{prop1}, we obtain
\begin{equation*}
a_{0}\left(1-\frac{q^{+}}{ \psi}\right) \|u_{n}\|^{p^{-}} \leq c_{9}(1+\|u_{n}\|).
\end{equation*}
This implies that $\{u_{n}\}$ is bounded in $W^{1,\mathcal{H}}_{0}(\Omega)$. As $W^{1,\mathcal{H}}_{0}(\Omega)$ is a reflexive space, $\exists \  u_{0}\in W^{1,\mathcal{H}}_{0}(\Omega)$ such that  up to a subsequence, we have $u_{n}\rightharpoonup u_{0}$ weakly in $W^{1,\mathcal{H}}_{0}(\Omega)$. Further, by Proposition \ref{prop4}, we have $u_{n}(x)\rightarrow u_{0}(x)$ a.e. $x\in\Omega$.
\end{proof}

 \begin{lemma}\label{l4}
  There exist $\lambda_{1}>1$  such that for each $\lambda\geq\lambda_{1},$ we have
 $$c_{M}<a_{0}\left(1-\frac{q^{+}}{ \psi} \right)\min\left\lbrace\frac{1}{N} \left( \frac{\alpha_{N}}{2^{\frac{N\alpha}{N-1}}}\right) ^{N-1},\frac{a_{\min}}{p^{+}}S^{N}  \right\rbrace,$$ 
 where $ \psi=\min\{\chi,\beta,(p^{*})^{-}\}$ and  $a_{\min}=\min\limits_{x\in\Omega} a_{0}^{({{N/ p(x)}})-1}$. Moreover, for  any $(PS)_{c_{M}}$ sequence $\{u_{n}\}$, we have
 $$\limsup_{n\rightarrow\infty}\|\nabla u_{n}\|_{L^{N}(\Omega_{N})}^{N/(N-1)}<\frac{\alpha_{N}}{2^{\frac{N\alpha}{N-1}}}\cdot$$
 \end{lemma}
 \begin{proof}
If $0\leq s\leq t_{1}=\max\{t_{0}^{q^{+}}, t_{0}^{p^{-}}\},$ then $A(s)\leq a(t_{1})s$, where $t_{0}$ is defined in the proof of the Lemma \ref{l2} . Let $\nu_{0}\in C_{c}^{\infty}(\Omega_{N}\backslash \overline{(\Omega_{q})_{\delta}})\backslash\{0\}$ be as in the Lemma \ref{l2} and $0\leq t\leq 1.$ By  using $(\mathcal{H}_{6})$, \eqref{theta}  and Proposition \ref{prop1}, we have
\begin{equation*} 
\begin{split}
J(t\nu_{0})& =A\left(m( t\nu_{0})\right)  -\int_{\Omega} F(x,t\nu_{0})dx \\
 & \leq {a(t_{1})}\left( m(t\nu_{0})\right)-\frac{\lambda |t|^{\beta}}{\beta}\int_{\Omega_{N}}|\nu_{0}|^{\beta}dx, \\
 & \leq \frac{a(t_{1})t^{N}C_{1}}{N}\left( \int_{\Omega_{N}}|\nabla \nu_{0}|^{N}\right)-\frac{\lambda |t|^{\beta}}{\beta}\int_{\Omega_{N}}|\nu_{0}|^{\beta}dx.
\end{split}
\end{equation*}
Further, we get
$$\max_{0\leq t\leq 1}J(t\nu_{0})\leq \frac{1}{\lambda^{\frac{N}{\beta-N}}}\left( \frac{1}{N}-\frac{1}{\beta}\right) \dfrac{\left( a(t_{1})C_{1}\|\nabla \nu_{0}\|_{L^{N}(\Omega_{N})}^{N}\right) ^{\frac{\beta}{\beta-N}}}{\left( \|\nu_{0}\|_{L^{\beta}(\Omega_{N})}^{\beta}\right) ^{\frac{N}{\beta-N}}}\cdot$$
On taking, $\gamma=t\nu_{0}$ for $0\leq t\leq 1$, we have
$$c_{M}\leq \max_{0\leq t\leq 1}J(t\nu_{0})\leq\frac{1}{\lambda^{\frac{N}{\beta-N}}}\left( \frac{1}{N}-\frac{1}{\beta}\right) \dfrac{\left( a(t_{1})C_{1}\|\nabla \nu_{0}\|_{L^{N}(\Omega_{N})}^{N}\right) ^{\frac{\beta}{\beta-N}}}{\left( \|\nu_{0}\|_{L^{\beta}(\Omega_{N})}^{\beta}\right) ^{\frac{N}{\beta-N}}}\cdot$$
Now, choosing $\lambda_{1}>1$ in such a way that $\forall \ \lambda\geq\lambda_{1},$ we have

\begin{equation*} 
\begin{split}
\frac{1}{\lambda^{\frac{N}{\beta-N}}}&\left( \frac{1}{N}-\frac{1}{\beta}\right) \dfrac{\left( a(t_{1})C_{1}\|\nabla \nu_{0}\|_{L^{N}(\Omega_{N})}^{N}\right) ^{\frac{\beta}{\beta-N}}}{\left( \|\nu_{0}\|_{L^{\beta}(\Omega_{N})}^{\beta}\right) ^{\frac{N}{\beta-N}}}\\
&<a_{0}\left(1-\frac{q^{+}}{ \psi} \right)\min\left\lbrace\frac{1}{N} \left( \frac{\alpha_{N}}{2^{\frac{N\alpha}{N-1}}}\right) ^{N-1},\frac{a_{\min}}{p^{+}}S^{N}  \right\rbrace\cdot
\end{split}
\end{equation*}
Therefore,
\begin{equation}\label{p33}
 c_{M}<a_{0}\left(1-\frac{q^{+}}{ \psi} \right)\min\left\lbrace\frac{1}{N} \left( \frac{\alpha_{N}}{2^{\frac{N\alpha}{N-1}}}\right) ^{N-1},\frac{a_{\min}}{p^{+}}S^{N}  \right\rbrace, \ \forall \lambda\geq \lambda_{1}.
\end{equation}
Moreover,  since $\{u_{n}\}$ is a $(PS)_{c_{M}}$ sequence for $J$, we have
$J(u_{n})\rightarrow c_{M}$ and $J'(u_{n})\rightarrow 0$ as  $n\rightarrow \infty.$ By \eqref{p11} and \eqref{p22}, we get
\begin{equation*} 
\begin{split}
A\left( m(u_{n})\right)-\frac{1}{ \psi}  a\left( m(u_{n})\right) &\int_{\Omega} h(x,|\nabla u_{n}|) \ |\nabla u_{n}|^{2} \\
&\leq \delta_{n}+c_{M}+\varepsilon_{n}\|u_{n}\|.
\end{split}
\end{equation*}
  It follows from $(a_{1})$  that
\begin{equation*}
a_{0}\left(1-\frac{q^{+}}{ \psi}\right) m(u_{n})
\leq \delta_{n}+c_{M}+\varepsilon_{n}\|u_{n}\|. 
\end{equation*}
By $ (\mathcal{H}_{6}) $ and \eqref{p33}, we get
 $$\limsup_{n\rightarrow\infty}\|\nabla u_{n}\|_{L^{N}(\Omega_{N})}^{N/(N-1)}\leq c_{M}<\frac{\alpha_{N}}{2^{\frac{N\alpha}{N-1}}}\cdot$$
\end{proof}

%**************************************************
 \begin{lemma}\label{l5}
   The functional $J$ satisfies the $(PS)_{c_{M}}$ condition.
  \end{lemma}

\begin{proof}
 Define
 $$P_{n}= a\left( m(u_{n})\right) \int_{\Omega} h(x,|\nabla u_{n}|) \ \nabla u_{n} \ \nabla (u_{n}-u).$$
Then
 $$P_{n}=J'(u_{n})u_{u}+ \int_{\Omega} f(x,u_{n}) u_{n} dx-J'(u_{n})u- \int_{\Omega} f(x,u) u dx.$$\\
  Using the definition of $f$, $P_{n}$ can be rewritten as 
  \begin{equation*}
 \begin{split}
  P_{n}=&\int_{\Omega}\widetilde{ \psi}_{q}(x)\varphi(x,u_{n})u_{n}dx+\int_{\Omega_{N}}\lambda|u_{n}|^{\beta}e^{\alpha|u_{n}|^{\frac{N}{N-1}}}dx+\int_{\Omega\backslash\Omega_{N}}\lambda \psi_{N}(x)|u_{n}|^{\beta}e^{\alpha|u_{n}|^{\frac{N}{N-1}}}dx\\
  &+\int_{\Omega_{p}}|u_{n}|^{p^{*}(x)}dx+\int_{\Omega\backslash\Omega_{p}} \psi_{p}(x)| u_{n}|^{p^{\ast}(x)}dx-\int_{\Omega}\widetilde{ \psi}_{q}(x)\varphi(x,u_{n})udx\\
  &-\int_{\Omega_{N}}\lambda|u_{n}|^{\beta-2}u_{n}e^{\alpha|u_{n}|^{\frac{N}{N-1}}}udx-\int_{\Omega\backslash\Omega_{N}}\lambda \psi_{N}(x)|u_{n}|^{\beta-2}u_{n}e^{\alpha|u_{n}|^{\frac{N}{N-1}}}udx\\
  &-\int_{\Omega_{p}}|u_{n}|^{p^{*}(x)-2}u_{n}u dx-\int_{\Omega\backslash\Omega_{p}} \psi_{p}(x)| u_{n}|^{p^{\ast}(x)-2} u_{n}udx+o_{n}(1).
  \end{split}
  \end{equation*}
  From the embedding results, we have 
   \begin{equation*}
   \begin{split}
    P_{n}=&\lambda\int_{\Omega_{N}}|u_{n}|^{\beta}e^{\alpha|u_{n}|^{\frac{N}{N-1}}}dx
    +\int_{\Omega_{p}}|u_{n}|^{p^{*}(x)}dx\\
    &-\lambda\int_{\Omega_{N}}|u_{n}|^{\beta-2}u_{n}e^{\alpha|u_{n}|^{\frac{N}{N-1}}}udx
    -\int_{\Omega_{p}}|u_{n}|^{p^{*}(x)-2}u_{n}u dx+o_{n}(1).
    \end{split}
    \end{equation*}
 As proved in the \cite[ Lemma 3.13]{double}, we have 
  $$P_{n}=\int_{\Omega_{p}}|u_{n}|^{p^{*}(x)}dx  -\int_{\Omega_{p}}|u_{n}|^{p^{*}(x)-2}u_{n}u dx+o_{n}(1).$$ \\
  By \cite[Theorem 1.14]{fan} and \cite[Theorem 3.1]{converg}, one gets
  $$P_{n}=\int_{\Omega_{p}}|u_{n}|^{p^{*}(x)}dx  -\int_{\Omega_{p}}|u|^{p^{*}(x)} dx+o_{n}(1).$$ \\
   Next, we will apply the Lemma \ref{cc} to the sequence $\{u_{n}\}\subset W^{1,p(x)}_{0}(\Omega_{p})$ and will prove that
\begin{equation}\label{ps5}
\int_{\Omega_{p}}|u_{n}|^{p^{*}(x)}dx\rightarrow\int_{\Omega_{p}}|u|^{p^{*}(x)}dx.   
\end{equation}
Since,  the $W^{1,\mathcal{H}}(\Omega)\hookrightarrow C(\overline{(\Omega_{q})_{\delta}})$ is compact and $\{u_{n}\}$ is bounded in $W^{1,\mathcal{H}}(\Omega)$, we get $u_{n}\rightarrow u$ in $L^{p^{*}(x)}((\Omega_{q})_{\delta}),$ which implies that $x_{i}\in\overline{\Omega_{p}}\backslash(\Omega_{q})_{\delta}$ for each $i\in I.$ To prove \eqref{ps5}, it is suffices to prove that $I$ is finite. Further, the set $I$ can be partitioned as 
$I=I_{1}\cup I_{2},$
where $I_{1}=\left\lbrace i\in I :  x_{i}\in{\Omega_{p}}\cap\partial(\Omega_{q})_{\delta}\right\rbrace\hbox{ and } I_{2}=\{ i\in I  :  x_{i}\in\overline{\Omega_{p}}\backslash\overline{(\Omega_{q})_{\delta}} \}.$
First, we show that $I_{1}$ is finite. Choose a cutoff function $v_{0}\in C_{c}^{\infty}(\R^{N})$ such that $$v_{0}\equiv 1 \ \text{on} \ {B}(0,1), \ v_{0}\equiv 0 \ \text{on} \ {B}(0,2)^{c}.$$  
 Now, for each $\epsilon>0,$ define $v(x)=v_{0}((x-x_{i})/\epsilon)$ $\forall x\in \R^{N}$.
  As $\{u_{n}\}$ is a $(PS)_{c_{M}}$ sequence, we have
\begin{equation*}
\begin{split} a\left( m(u_{n})\right) &\int_{\Omega} h(x,|\nabla u_{n}|) \ \nabla u_{n} \  \nabla (v u_{n})=\int_{\Omega}\widetilde{ \psi}_{q}(x)\varphi(x,u_{n})v u_{n}dx\\
&+\int_{\Omega} \psi_{p}(x)| u_{n}|^{p^{\ast}(x)}v dx+o_{n}(1).
\end{split} 
\end{equation*}
It follows from $(a_{1})$ that
 $$ a_{0}\int_{\Omega} h(x,|\nabla u_{n}|) \ \nabla u_{n} \ \nabla (v u_{n})dx\leq\int_{\Omega}\widetilde{ \psi}_{q}(x)\varphi(x,u_{n})v u_{n}dx+\int_{\Omega} \psi_{p}(x)| u_{n}|^{p^{\ast}(x)}v dx+o_{n}(1),$$
 which implies that
\begin{equation}\label{ps3}
   \begin{split}
    a_{0}&\int_{\Omega} h(x,|\nabla u_{n}|) \ u_{n} \nabla u_{n} \ \nabla v dx\leq\int_{\Omega}\widetilde{ \psi}_{q}(x)\varphi(x,u_{n})v u_{n}dx\\
    &+\int_{\Omega} \psi_{p}(x)| u_{n}|^{p^{\ast}(x)}v dx-a_{0}\int_{\Omega} h(x,|\nabla u_{n}|) \ |\nabla u_{n}|^{2} \ v dx+o_{n}(1).
    \end{split}
    \end{equation}
Next, by using  \eqref{comp},    $\Delta_{2}$-condition and Young's inequality, we get
\begin{equation}\label{ps1}
\int_{\Omega}| h(x,|\nabla u_{n}|)| \  |\nabla u_{n}| |u_{n}| \ |\nabla v| dx\leq \zeta m(u_{n})+C_{\zeta}\int_{\Omega}\ \mathcal{H}(x,|u_{n}| \ |\nabla v|)dx.
\end{equation}

On using $ (\mathcal{H}_{8}) $, one gets
$$\int_{\Omega}\ \mathcal{H}(x,|u_{n}| \ |\nabla v|)dx\leq c_{10}\left( \int_{\Omega} \eta_{2}(x)|u_{n}|^{q(x)} \ |\nabla v|^{q(x)}dx+ \int_{\Omega}|u_{n}|^{p(x)} \ |\nabla v|^{p(x)}dx\right),$$  
for some $c_{10}>0$. Using generalized  H$\ddot{\text{o}}$lder's inequality \ref{prop5}, we get
\begin{align*}
\int_{\Omega}\ \mathcal{H}(x,|u_{n}| \ |\nabla v|)dx&\leq c_{11}\| |\nabla v|^{q(x)}\|_{L^{\frac{p^{*}(x)}{p^{*}(x)-q(x)}}(\Omega)}\||u_{n}|^{q(x)}\|_{L^{\frac{p^{*}(x)}{q(x)}}(\Omega)}\\
&+c_{12} \| |\nabla v|^{p(x)}\|_{L^{\frac{N}{p(x)}}(\Omega)}\||u_{n}|^{p(x)}\|_{L^{\frac{N}{N-p(x)}}(\Omega)},
\end{align*}
for some $c_{11},c_{12}>0$. Further, by Proposition \ref{prop2} and Proposition \ref{prop6}, we have
\begin{align*}
\int_{\Omega}\ \mathcal{H}(x,|u_{n}| \ |\nabla v|)dx&\leq c_{11}\max\left\lbrace\left( \int_{{B}(x_{i},2\epsilon)}|\nabla v|^{\frac{q(x)p^{*}(x)}{p^{*}(x)-q(x)}}\right)^{p_{1}}, \left(\int_{{B}(x_{i},2\epsilon)}|\nabla v|^{\frac{q(x)p^{*}(x)}{p^{*}(x)-q(x)}}\right)^{p_{2}}\right\rbrace\\ 
&\max\left\lbrace \|u_{n}\|^{q^{-}}_{L^{p^{*}(x)}(\Omega)},\|u_{n}\|^{q^{+}}_{L^{p^{*}(x)}(\Omega)}\right\rbrace \\
&+c_{12}\max\left\lbrace  \left( \int_{{B}(x_{i},2\epsilon)}|\nabla v|^{N}\right)^{p^{-}/N},  \left( \int_{{B}(x_{i},2\epsilon)}|\nabla v|^{N}\right)^{p^{+}/N}\right\rbrace\\
&\max\left\lbrace \|u_{n}\|^{p^{-}}_{L^{p^{*}(x)}(\Omega)},\|u_{n}\|^{p^{+}}_{L^{p^{*}(x)}(\Omega)}\right\rbrace,
\end{align*}
where $p_{1}=\min\limits_{x\in\Omega}\left\lbrace \frac{p^{*}(x)-q(x)}{p^{*}(x)}\right\rbrace $ and $p_{2}=\max\limits_{x\in\Omega}\left\lbrace \frac{p^{*}(x)-q(x)}{p^{*}(x)}\right\rbrace \cdot$

Hence, by using Proposition \ref{prop3}$(b)$ together with the fact that $\{u_{n}\}$ is bounded, we have
\begin{equation}\label{ps2}
\lim_{\epsilon\rightarrow 0}\int_{\Omega}\ \mathcal{H}(x,|u_{n}| \ |\nabla v |)=0.
\end{equation}
By using \eqref{ps1}, \eqref{ps2}, Proposition \ref{prop1} and using the boundedness of $\{u_{n}\}$, we get
\begin{equation*}
\lim_{\epsilon\rightarrow 0}\lim_{n\rightarrow\infty}\int_{\Omega}| h(x,|\nabla u_{n}|)| \  |\nabla u_{n}| |u_{n}| \ |\nabla v| dx\leq c_{12}\zeta ,
\end{equation*} for some $c_{12}>0$.
As $\zeta$ is arbitrary, one get
\begin{equation}\label{ps4}
\lim_{\epsilon\rightarrow 0}\left(\lim_{n\rightarrow\infty} \int_{\Omega}| h(x,|\nabla u_{n}|)| \  |\nabla u_{n}| |u_{n}| \ |\nabla v| dx\right)=0.
\end{equation}
Consequently, by $ (\mathcal{H}_{8}) $, \eqref{ps3} and \eqref{ps4}, we get
\begin{equation*}
 \lim_{\epsilon\rightarrow 0}\lim_{n\rightarrow\infty}\left( \int_{\Omega}\widetilde{ \psi}_{q}(x)\varphi(x,u_{n})v u_{n}dx
    +\int_{\Omega} \psi_{p}(x)| u_{n}|^{p^{\ast}(x)}v dx-a_{0}\int_{\Omega_{p}}|\nabla u_{n}|^{p(x)} \ v dx\right) \geq   0,
\end{equation*}
which is 
\begin{equation*}
 \nu_{i}-a_{0}\mu_{i}=\lim_{\epsilon\rightarrow 0}\left(
    \lim_{n\rightarrow\infty}\int_{\Omega_{p}}| u_{n}|^{p^{\ast}(x)}v dx-a_{0}\lim_{n\rightarrow\infty}\int_{\Omega_{p}}|\nabla u_{n}|^{p(x)} \ v dx\right) \geq   0.
\end{equation*}
By Lemma \ref{cc}, we get $\nu_{i}\geq a_{0}S^{p(x_{i})}\nu_{i}^{p(x_{i})/p^{*}(x_{i})}$, consequently either $\nu_{i}=0$ or $\nu_{i}\geq a_{0}^{\frac{N}{p(x_{i})}}S^{N}$. Next, we will prove that $\nu_{i}\geq a_{0}^{\frac{N}{p(x_{i})}}S^{N}$ is not possible. Let suppose $\nu_{i}\geq a_{0}^{\frac{N}{p(x_{i})}}S^{N},$ then by Lemma \ref{cc}, we get $\mu_{i}\geq S^{N}a_{0}^{{\frac{N}{p(x_{i})}}-1}$. Also, since $|\nabla u_{n}|^{p(x)}$ converges weakly to a measure $\mu,$ 
$$\liminf_{n\rightarrow\infty}\int_{\Omega_{p}}|\nabla u_{n}|^{p(x)}dx\geq \mu_{i},$$
and hence,
\begin{equation}\label{111}
\liminf_{n\rightarrow\infty}\int_{\Omega_{p}}|\nabla u_{n}|^{p(x)}dx\geq S^{N}a_{0}^{{\frac{N}{p(x_{i})}}-1}\geq S^{N}a_{\min},
\end{equation}
where, $a_{\min}=\min\limits_{x\in\Omega} a_{0}^{({{N/ p(x_{i})}})-1}$.

Since,  $\{u_{n}\}$ is a $(PS)_{c_{M}}$ sequence for $J$, we have
$J(u_{n})\rightarrow c_{M}$ and $J'(u_{n})\rightarrow 0$ as  $n\rightarrow \infty.$ By \eqref{p1}, \eqref{p11} and \eqref{p22}, we get
 \begin{equation*}
 \begin{split}
A\left( m(u_{n})\right)-\frac{1}{ \psi}a\left( m(u_{n})\right) &\int_{\Omega} h(x,|\nabla u_{n}|) \ |\nabla u_{n}|^{2} dx\\
&\leq \delta_{n}+c_{M}+\varepsilon_{n}\|u_{n}\|.
  \end{split}
 \end{equation*}

  It follows from $(a_{1})$  that
\begin{equation*}
{a_{0}}\left(1-\frac{q^{+}}{ \psi}\right) m(u_{n})
\leq \delta_{n}+c_{M}+\varepsilon_{n}\|u_{n}\|. 
\end{equation*}
By $ (\mathcal{H}_{8}) $ and \eqref{p33}, we obtain
 $$\frac{a_{0}}{p^{+}}\left(1-\frac{q^{+}}{ \psi}\right)\liminf_{n\rightarrow\infty}\int_{\Omega_{p}}|\nabla u_{n}|^{p(x)}dx\leq c_{M}< \frac{a_{0}}{p^{+}}\left(1-\frac{q^{+}}{ \psi}\right)  S^{N}a_{\min},$$
 which is a contradiction to \eqref{111}. Hence, $I_{1}$ is an empty set. 
 
  By using  $(\mathcal{H}_{8})$ and proceeding as above, one can show that $I_{2}=\emptyset.$ 

Therefore, we get  $P_{n}=o_{n}(1),$ and so $$\lim_{n\rightarrow\infty} a\left( m(u_{n})\right) \int_{\Omega} h(x,|\nabla u_{n}|) \ \nabla u_{n} \ \nabla (u_{n}-u)= 0,$$
from which we have 
$$\lim_{n\rightarrow\infty} \int_{\Omega} h(x,|\nabla u_{n}|) \ \nabla u_{n} \ \nabla (u_{n}-u)\leq 0.$$ 
By Lemma \ref{l8}, we have $u_{n}\rightarrow  u$ in $W^{1,\mathcal{H}}_{0}(\Omega)$.
\end{proof}

%++++++++++++++++++++++++++++++++++++++++++++++++++++++ 

Now, we are ready to prove the Theorem \ref{t1}. 

\textit{Proof of the Theorem \ref{t1}}.
By  Lemmas \ref{l1} and \ref{l2}, the geometric conditions of the mountain pass theorem are satisfied for the functional $J$ and by Lemma \ref{l4},  $(PS)_{c_{M}}$ condition is satisfied.
 Hence, by the  mountain pass theorem, $\exists$  a critical point $u_{M}$ of $J$ with level $c_{M}$, i.e., $J'(u_{M})=0$ and $J(u_{M})=c_{M}$. Thus, $u_{M}$ is the weak solution of the problem \eqref{1.1}.  \qed

%\section{Example}
%In order to give clear view of our hypotheses on the nonlinearity $f$, the nonlocal term $A$ and the perturbation term $h,$ we provide some examples. We can take $f$ as the following function.

\section*{Acknowledgement} The second author is supported by Science and Engineering Research Board, India, under the grant CRG/2020/002087.


\begin{thebibliography}{0}

\bibitem{adi} Adimurthi, Existence of positive solutions of the semilinear Dirichlet problems with critical growth for the $N$-Laplacian, Ann. Sc. Norm. Super. Pisa,
\textbf{17} (1990), 393--413.

\bibitem{Adimurthi} Adimurth, K. Sandeep,  A singular Moser-Trudinger embedding and its applications, Ann. Polon. Math., \textbf{13} (2007), 585--603.


\bibitem{alves2021} C. O. Alves,  T.  Boudjeriou, Existence of solution for a class of nonlocal problem via dynamical methods,  Rend. Circ. Mat. Palermo (2), https://doi.org/10.1007/s12215-021-00644-4.
\bibitem{alves2001}  C. O. Alves, F. J. S. A. Corr\^{e}a, On existence of solutions for a class of problem involving a nonlinear operator. Comm. Appl. Nonlinear Anal., \textbf{2} (2001), 43--56.
\bibitem{alves2005} C. O. Alves, F. J. S. A. Corr\^{e}a,  T. F. Ma, Positive solutions for a quasilinear elliptic equation of Kirchhoff type, Comput. Math. Appl., \textbf{49} (2005), 85--93.

\bibitem{double}  C. O. Alves, P.  Garain, V. D. R$\breve{\text{a}}$dulescu, High perturbations of quasilinear problems with double criticality, Math. Z., https://doi.org/10.1007/s00209-021-02757-z.
\bibitem{cri} J. G. Azorero, I. P. Alonso,  Multiplicity of solutions for elliptic problems with critical exponent or with a nonsymmetric term, Trans. Am. Math. Soc., \textbf{2} (1991)  877--895.

\bibitem{silva} J. F. Bonder,  A.  Silva,  Concentration-compactness  principle  for  variable  exponent  spaces  and  applications,  Electron.  J. Differential Equations, \textbf{141} (2010) 1--18.

\bibitem{razani} M. M. Chaharlang, A. Razani, Existence of infinitely many solutions for a class of nonlocal problems with Dirichlet boundary condition, Commun. Korean Math. Soc., \textbf{34} (2019) 155--167.

\bibitem{chung1} N. T. Chung, Multiplicity results for a class of $p(x)$-Kirchhoff type equations with combined nonlinearities, Electron. J. Qual. Theory Differ. Equ., \textbf{42} (2012) 1--13.
\bibitem{Chung2} N. T. Chung, Three solutions for a class of nonlocal problems in Orlicz-Sobolev spaces, J. Korean Math. Soc., \textbf{50} (2013)  1257--1269.




\bibitem{chipot1997some} M. Chipot, B. Lovat, Some remarks on non local elliptic and parabolic problems, Nonlinear Anal., \textbf{30} (1997), 4619--4627.
\bibitem{pocket} I.  Chlebicka, A pocket guide to nonlinear differential equations in Musielak-Orlicz spaces, Nonlinear Anal., \textbf{175} (2018), 1--27.
\bibitem{p_x} F. Colasuonno, P. Pucci, Multiplicity of solutions for $p(x)$-polyharmonic elliptic Kirchhoff equations,  Nonlinear Anal., \textbf{74} (2011), 5962--5974.

\bibitem{cruz2013variable} D. V. Cruz-Uribe, A. Fiorenza,  Variable Lebesgue spaces: Foundations and harmonic analysis, Springer Science \& Business Media, (2013).

\bibitem{dai}  G. Dai, R. Hao, Existence of solutions for a $p(x)$-Kirchhoff-type equation, J. Math. Anal. Appl., \textbf{359} (2009) 275--284.
\bibitem{dai1} G. Dai, D. Liu, Infinitely many positive solutions for a $p(x)$-Kirchhoff-type equation, J. Math. Anal. Appl., \textbf{359} (2009) 704--710.

\bibitem{do2} J. M. B. do $\acute{\hbox{O}}$, $N$-Laplacian equations in $\R^{N}$ with critical growth, Abstr. Appl. Anal., \textbf{2} (1997), 301--315.

\bibitem{do1} J. M. B. do $\acute{\hbox{O}}$, E. S. Medeiros, U. Severo, On a quasilinear non homogeneous elliptic equation with critical growth in $\R^{N}$, J. Differential Equations, \textbf{246} (2009), 1363--1386.
\bibitem{enmunds} D. Edmunds, J. Rakosnik, Sobolev embeddings with variable exponent,  Studia Math., \textbf{143} (2000), 267--293.
\bibitem{subsup} X.  Fan,  Differential equations of divergence form in Musielak-Sobolev spaces and a sub-supersolution method,  J. Math. Anal. Appl., \textbf{386} (2012) 593--604.
\bibitem{fan}  X. Fan,  D. Zhao, On the Spaces  $L^{p(x)}(\Omega)$  and  $W^{m,p(x)}(\Omega)$, J. Math. Anal. Appl., \textbf{236} (2001), 424--446.

\bibitem{figu2019} G. M.  Figueiredo, F. B. M. Nunes, Existence of positive solutions for a class of quasilinear elliptic problems with exponential growth via the Nehari manifold method,  Rev. Mat. Complut., \textbf{32} (2019), 1--18.


\bibitem{fukagai2006} N. Fukagai, M. Ito, K.  Narukawa, Positive solutions of quasilinear elliptic equations with critical Orlicz-Sobolev nonlinearity on $\R^N$, Funkcial. Ekvac., \textbf{46} (2006), 235--267.
\bibitem{furtado2019} M. F.  Furtado, L. D. de Oliveira, J. P. P. da Silva, Multiple solutions for a Kirchhoff equation with critical growth, Z. Angew. Math. Phys., \textbf{70} (2019),  1--15.

\bibitem{gossez} J. Gossez, Nonlinear elliptic boundary value problems for equations with rapidly (or slowly) increasing coefficients, Trans. Am. Math. Soc., \textbf{190} (1974), 163–-205. 
\bibitem{goyal} S. Goyal, P. K.  Mishra, K. Sreenadh, $n$-Kirchhoff type equations with exponential nonlinearities, RACSAM,  \textbf{116} (2016), 219--245.

\bibitem{hamdani} M. K. Hamdani, A. Harrabi, F. Mtiri, D. D. Repov$\hat{\text{s}}$, Existence and multiplicity results for a new $p(x)$-Kirchhoff problem, Nonlinear Anal., \textbf{190} (2020) 111598.

\bibitem{hamydy2011existence} A. Hamydy, M. Massar, N. Tsouli, Existence of solutions for p-Kirchhoff type problems with critical exponent, Electron. J. Differential Equations, \textbf{105} (2011), 1--8.


\bibitem{hudzik} H. Hudzik, On generalized Orlicz-Sobolev space, Funct. Approx., \textbf{4} (1977), 37--51.

\bibitem{hudzik1} H. Hudzik, The  problems  of  separability,  duality,  reflexivity  and  of  comparison  for generalized Orlicz-Sobolev space $W_{M}^{k}(\Omega)$, Comment. Math., \textbf{21} (1979), 315--324.
\bibitem{converg} Y. Kaya,  A Weakly Convergence Result on $L^{p(x)}$ Spaces,  Math. Comput. Appl., \textbf{20} (2015) 106--110.
\bibitem{kesawan} S. Kesavan,  Topics in functional analysis and applications, Wiley, (1989).
\bibitem{kirchhoff1876vorlesungen} G. Kirchhoff, Vorlesungen {\"u}ber mathematische physik: mechanik, \textbf{1},  BG Teubner, (1876).
\bibitem{kovacik} O. Kovacik, J. Rakosnik, On spaces $L^{p(x)}(\Omega)$  and  $W^{k,p(x)}(\Omega)$,  Czechoslovak Math. J., \textbf{41} (1991), 592--618.
\bibitem{Lam} N.  Lam, G.  Lu,  Existence and multiplicity of solutions to equations of $N$-Laplacian type with critical exponential growth in $\R^N$, J. Funct. Anal., \textbf{262} (2012), 1132--1165.
\bibitem{lam2014} N. Lam, G. Lu, Elliptic equations and systems with subcritical and critical exponential growth without the Ambrosetti-Rabinowitz condition, J. Geom. Anal., \textbf{24} (2014),  118--143.

\bibitem{lee1} J. Lee, J. M. Kim, Y. H. Kim, Existence and multiplicity of solutions for Kirchhoff-Schr$\ddot{\text{o}}$dinger type equations involving $p(x)$-Laplacian on the entire space $\R^{N}$, Nonlinear Anal. Real World Appl., \textbf{45} (2019) 620--649.

\bibitem{lions1978} J. L. Lions, On some questions in boundary value problems of mathematical physics,  North-Holland Math. Stud., \textbf{30}  (1978), 284--346.
\bibitem{plion}  P. L. Lions, The concentration-compactness principle in the calculus of variations. The limit case, part 1, Rev. Mat. Iberoamericana, \textbf{1} (1985), 145-201.
\bibitem{liu2015} D. Liu, P. Zhao, Solutions for a quasilinear elliptic equation in Musielak-Sobolev spaces, Nonlinear Anal. Real World Appl., \textbf{26}  (2015), 315--329.
\bibitem{hachimi} S. Maatouk, A. El Hachimi, Quasilinear elliptic problem without Ambrosetti-Rabinowitz condition involving a potential in Musielak-Sobolev spaces setting, Complex Var. Elliptic Equ., \textbf{66} (2021), 2028--2054.
\bibitem{moser} J. Moser, A sharp form of an inequality by N. Trudinger, Indiana Univ. Math. J., \textbf{20} (1971), 1077--1092.
\bibitem{musielak} J. Musielak, Orlicz Spaces and Modular Spaces, Lecture Notes in Mathematics, Springer-Verlag, Berlin etc., (1983).
\bibitem{naiman2017} D. Naimen, C. Tars, Multiple solutions of a Kirchhoff type elliptic problem with the Trudinger-Moser growth, Adv. Difference Equ.,  \textbf{22} (2017), 983--1012.
\bibitem{radulescu} V. D.  R$\breve{\text{a}}$dulescu, D. Vicentiu,  D. D. Repovs,  Partial differential equations with variable exponents: variational methods and qualitative analysis, CRC press,  \textbf{9}  (2015).
\bibitem{shi} Z. Shi, S. Wu, Existence of solutions for Kirchhoff type problems in Musielak-Orlicz-Sobolev spaces, J. Math. Anal. Appl., \textbf{436} (2016) 1002--1016.
\bibitem{sun2016} J. Sun, Y. Ji, T. F. Wu, Infinitely many solutions for Kirchhoff-type problems depending on a parameter, Electron. J. Differential Equations, \textbf{224} (2016), 1--10.
\bibitem{trudinger} N. S. Trudinger, On imbeddings into Orlicz spaces and some applications, J. Math. Mech., \textbf{17} (1977), 473--483.

\bibitem{wang2018} L. Wang, K. Xie, B. Zhang, Existence and multiplicity of solutions for critical  Kirchhoff-type $p$-Laplacian problems, J. Math. Anal. Appl., \textbf{458} (2018), 361--378.
\bibitem{radulescu1} M. Xiang, B. Zhang, V. D.  R$\breve{\text{a}}$dulescu,  Superlinear Schrödinger-Kirchhoff type problems involving the fractional $p$-Laplacian and critical exponent, Adv. Nonlinear Anal.,  \textbf{9}  (2020), 690--709.

\bibitem{yan2021} B. Yan, D. O'Regan,  R. P. Agarwal, Infinite number of solutions for some elliptic eigenvalue problems of Kirchhoff-type with non-homogeneous material, Bound. Value Probl., \textbf{1} (2021), 1--15.





\end{thebibliography}
\end{document}